\documentclass{amsart}

\usepackage[T1]{fontenc}
\usepackage[latin1]{inputenc}
\usepackage[english]{babel}
\usepackage{amsmath}
\usepackage{amsfonts}
\usepackage{amssymb}
\usepackage{amsthm}
\usepackage{enumitem}
\usepackage{hyperref}

\newtheorem{thm}{Theorem}[section]
\newtheorem{lem}[thm]{Lemma}
\newtheorem{prop}[thm]{Proposition}
\newtheorem{cor}[thm]{Corollary}

\theoremstyle{definition}
\newtheorem{defn}[thm]{Definition}

\theoremstyle{remark}
\newtheorem{rem}[thm]{Remark}
\newtheorem{exa}[thm]{Example}
\newtheorem*{rem*}{Remark}
\newtheorem*{exa*}{Example}

\newcommand{\N}{\mathbb N}
\newcommand{\C}{\mathbb C}
\newcommand{\polyring}[1]{\mathbb C[z_1,\ldots,z_{#1}]}
\newcommand{\polyrin}{\mathbb C[z]}
\newcommand{\F}{\mathcal F}
\newcommand{\A}{\mathcal A}
\newcommand{\D}{\mathcal D}
\newcommand{\K}{\mathcal K}
\newcommand{\B}{\mathbb B}
\DeclareMathOperator{\spa}{span}

\DeclareMathOperator{\ran}{ran}
\newcommand{\tfa}{\text{ for all }}
\newcommand{\tand}{\text{ and }}

\newcommand{\compmap}[1]{C_{ {#1}^*}}
\newcommand{\fock}[1]{\F(#1)}
\newcommand{\symfock}[1]{\F_s(#1)}
\newcommand{\fangle}{c}
\newcommand{\admideal}{admissible}
\newcommand{\clop}{\mathcal B}

\author{Michael Hartz}
\title[Isomorphisms for operator algebras]{Topological isomorphisms for some universal operator algebras}
\date{}
\address{Fachrichtung Mathematik, Universit\"at des Saarlandes, Postfach 151150, 66041 Saarbr\"ucken, Germany}
\email{hartz@math.uni-sb.de}

\begin{document}

\begin{abstract}
  Let $I \subset \mathbb C[z_1,\ldots,z_d]$ be a radical homogeneous ideal,
  and let $\mathcal A_I$ be the norm-closed non-selfadjoint algebra
  generated by the compressions of the $d$-shift on Drury-Arveson space
  $H^2_d$ to the co-invariant subspace $H^2_d \ominus I$.
  Then $\mathcal A_I$ is the universal operator algebra
  for commuting row contractions subject to the relations in $I$.
  We ask under which conditions are there topological isomorphisms
  between two such algebras $\mathcal A_I$ and $\mathcal A_J$?
  We provide a positive answer to a conjecture of Davidson, Ramsey and Shalit:
  $\mathcal A_I$ and $\mathcal A_J$ are topologically isomorphic if
  and only if there is an invertible linear map $A$ on $\mathbb C^d$
  which maps the vanishing locus of $J$ isometrically onto
  the vanishing locus of $I$.
  Most of the proof is devoted to showing
  that finite algebraic sums
  of full Fock spaces over subspaces of $\mathbb C^d$ are closed.
  This allows us to show that the map $A$ induces a
  completely bounded
  isomorphism between $\mathcal A_I$ and $\mathcal A_J$.
\end{abstract}

\maketitle

\section{Introduction}
  Let $H^2_d$ be the Drury-Arveson space, that is, the reproducing kernel Hilbert space
  on the open unital ball in $\C^d$ with reproducing kernel
  \begin{equation*}
    K(z,w) = \frac{1}{1-\langle z,w \rangle},
  \end{equation*}
  also known as symmetric Fock space.
  The coordinate functions $z_i$ are multipliers on $H^2_d$, that is,
  the multiplication operators
  \begin{equation*}
    M_{z_i} : H^2_d \to H^2_d, \quad f \mapsto z_i f
  \end{equation*}
  determine a commuting operator tuple
  $S = M_z = (M_{z_1}, \ldots , M_{z_d})$,
  which is known as the \emph{d-shift}. The tuple $S$ is a row contraction, and according
  to Arveson \cite{arv-s3} (see also \cite{popescu-isometries} and \cite{ep02}), the unital non-selfadjoint norm-closed algebra $\A_d$
  generated by $M_{z}$ is universal for row contractions,
  in the sense that whenever $T = (T_1,\ldots,T_d)$ is any commuting row contraction on a Hilbert space
  $H$, the algebra homomorphism
  \begin{equation*}
    \polyring{d} \to \clop(H), \quad p \mapsto p(T_1, \ldots, T_d)
  \end{equation*}
  extends to a completely contractive representation of $\A_d$.

  Recently, Davidson, Ramsey and Shalit \cite{davramshal} examined
  universal operator algebras
  for commuting row contractions which satisfy relations given by a homogeneous ideal
  $I \subset \polyring{d}$.
  Algebras of this type, even in a more general case, were already studied by Popescu \cite{popescu}.
  The universal object in this setting is the quotient
  algebra $\A_d / \overline{I}$,
  which is an abstract operator algebra in the sense of Blecher, Ruan and Sinclair
  (see for example \cite[Chapter 17]{effros}).
  Popescu's work \cite{popescu} shows that
  $\A_d / \overline{I}$ can also be identified with the concrete
  algebra of operators $\A_I$ obtained by compressing $\A_d$ to the co-invariant subspace
  \begin{equation*}
    \F_I = H^2_d \ominus I.
  \end{equation*}

  If $I$ is a radical homogeneous ideal, $\A_I$ can be regarded as an algebra
  of continuous functions on the intersection of the vanishing locus $V(I)$ of $I$ with
  the closed unit ball. In particular, $\A_I$ is a commutative semi-simple Banach algebra in this case.
  In \cite{davramshal}, the isomorphism problem for algebras $\A_I$ of this type,
  and non-commutative generalizations thereof, was investigated.
  In the commutative radical case, a close connection between the structure of the algebra $\A_I$ and the geometry
  of the vanishing locus $V(I)$ of $I$ was established.
  More precisely, the authors of \cite{davramshal} proved, building upon results
  due to Shalit and Solel \cite{shalit_solel},
  that for two radical homogeneous ideals
  $I$ and $J$ in $\polyring{d}$, the algebras $\A_I$ and $\A_J$ are completely
  isometrically isomorphic if and only if they are isometrically isomorphic, which in turn happens if and only
  if there is a unitary map $U$ on $\C^d$ mapping $V(I)$ onto $V(J)$.

  Moreover, Davidson, Ramsey and Shalit studied the existence of algebraic isomorphisms,
  which are the same as topological isomorphisms since the algebras $\A_I$ are semi-simple in the radical case.
  They showed that if $I \subset \polyring{d}$ and $J \subset \polyring{d'}$ are radical homogeneous ideals
  such that $\A_I$ and $\A_J$ are topologically isomorphic, then there exist two linear maps
  $A: \C^{d'} \to \C^{d}$ and $B: \C^{d} \to \C^{d'}$ which restrict to mutually inverse bijections $A: Z(J) \to Z(I)$
  and $B: Z(I) \to Z(J)$, where $Z(I) = V(I) \cap \overline{\B_d}$ and
  $Z(J) = V(J) \cap \overline{\B_{d'}}$.
  The converse of this fact was established in \cite{davramshal} for the case of
  tractable varieties,
  and was conjectured to be true in general. In fact, Davidson, Ramsey and Shalit reduced this problem
  to the case where $I$ and $J$ are vanishing ideals of unions of subspaces.
  To give an example, single subspaces and unions of two subspaces are always tractable. However, unions of three or more subspaces
  are not tractable in general.

  The aim of the present note is to
  prove the following theorem, which establishes the above conjecture in full generality.

\begin{thm}
  Let $I$ and $J$ be radical homogeneous ideals in $\polyring{d}$ and $\polyring{d'}$, respectively.
  The algebras $\mathcal A_I$ and $\mathcal A_J$ are isomorphic if and only if there exist linear maps
  $A: \C^{d'} \to \C^{d}$ and $B: \C^{d} \to \C^{d'}$ which restrict to mutually inverse
  bijections $A: Z(J) \to Z(I)$ and $B: Z(I) \to Z(J)$.
\end{thm}

  To this end, we proceed as follows.
  In Section 2, we show that to establish
  the conjecture, it is enough to prove
  that if $V_1, \ldots, V_r \subset \C^d$ are subspaces,
  then the algebraic sum of the full Fock spaces $\fock{V_1} + \ldots + \fock{V_r}$ is closed in
  $\fock{\C^d}$.
  This problem can be approached using the notion of the Friedrichs angle between subspaces
  of a Hilbert space. In Section 3, we recall the basic facts concerning this concept, and
  we introduce a variant of the Friedrichs angle using the Calkin algebra
  which is more suitable to our needs than the classical one.
  The main result of Section 4 is Lemma \ref{lem:angle_several_sum_perp}, which
  reduces the problem of showing closedness of $\fock{V_1} + \ldots + \fock{V_r}$
  to the case where $V_1 \cap \ldots \cap V_r = \{0\}$.
  Section 5 finally contains a proof of the closedness of $\fock{V_1} + \ldots + \fock{V_r}$.

\section{\texorpdfstring{Algebra isomorphisms and sums of Fock spaces}{Algebra isomorphisms and sums of Fock spaces}}

As usual, let $\polyring{d}$ denote the algebra of complex polynomials in $d$ variables. When $d$ is understood,
we will simply write $\polyrin$. If $n$ is a natural number, then $\polyrin_n$ will denote the space of homogeneous polynomials of
degree $n$.
For a radical homogeneous ideal $I \subset \polyrin$, let $\F_I = H^2_d \ominus I$ and let $\A_I \subset \clop(\F_I)$
be the norm-closed non-selfadjoint algebra generated by the compressions of $M_{z_i}$ to the co-invariant
subspace $\F_I$. The vanishing locus of $I$ will be denoted by $V(I)$, and
we will write $Z^0(I)$ (respectively $Z(I)$) for the intersection of
$V(I)$ with the open (respectively the closed) unit ball.
Moreover, for a subset $S$ of a vector space,
$\spa(S)$ will denote the linear span of $S$.

We follow the route of \cite{davramshal} and try to find isomorphisms between the Hilbert spaces
$\F_I$ such that conjugation with these isomorphisms yields algebra isomorphisms between the algebras
$\A_I$.
We begin by exhibiting a convenient generating set for the $n$th homogeneous part
$\F_I \cap \polyrin_n$
of $\F_I$ (compare the discussion preceding \cite[Lemma 7.11]{davramshal}).
\begin{lem}
  \label{lem:F_I_generating_set}
  Let $I \subset \polyrin$ be a radical homogeneous ideal. Then for all
  natural numbers $n$,
  \begin{equation*}
    \F_I \cap \polyrin_n
    = \spa \{ \langle \cdot,\lambda \rangle^n: \lambda \in Z^0(I) \}
    = \spa \{ \langle \cdot,\lambda \rangle^n: \lambda \in V(I) \}.
  \end{equation*}
\end{lem}

\begin{proof}
 Note that for any $\lambda \in \B_d$, we have
 \begin{equation*}
  K(\cdot,\lambda) = \sum_{n=0}^\infty \langle \cdot,\lambda \rangle^n \in H^2_d,
 \end{equation*}
 where $K$ is the reproducing kernel of $H^2_d$.
 Using that homogeneous polynomials of different degree are orthogonal
 in $H^2_d$, we obtain for $\lambda \in Z^0(I)$ and $f \in \polyrin_n$ the identity
 \begin{equation*}
  \Big\langle f, \langle \cdot,\lambda \rangle^n  \Big\rangle_{H^2_d}
  = \Big\langle f, K(\cdot,\lambda)   \Big\rangle_{H^2_d} = f(\lambda).
 \end{equation*}
 In particular, if $f \in I \cap \polyrin_n$ and $\lambda \in Z^0(I)$, then
 \begin{equation*}
  \Big\langle f, \langle \cdot, \lambda \rangle^n \Big\rangle_{H^2_d} = 0,
 \end{equation*}
 hence $\langle \cdot,\lambda \rangle^n \in \F_I$. Conversely, if
 $g \in \F_I \cap \polyrin_n$ is orthogonal to each
 $\langle \cdot,\lambda \rangle^n$
 for $\lambda \in Z^0(I)$, then $g$ vanishes on $Z^0(I)$. By homogeneity of $I$
 and $g$, we infer that $g$ vanishes on $V(I)$, hence $g \in I$ by Hilbert's Nullstellensatz.
 Consequently, $g=0$, from which the first equality follows, while the second is obvious.
\end{proof}

Suppose now that $I \subset \polyring{d}$ and $J \subset \polyring{d'}$ are radical homogeneous
ideals and that $A: \C^{d'} \to \C^{d}$ is a linear map which maps $V(J)$ into $V(I)$.
It is an easy consequence of the homogeneity of $J$ that $\D_J = \F_J \cap \polyring{d'}$ is a dense
subspace of $\F_J$.
Since
\begin{equation}
  \label{eqn:com_A_skal_prod}
  \langle \cdot,\lambda \rangle^n \circ A^* = \langle \cdot, A \lambda \rangle^n
\end{equation}
for all $\lambda \in \C^{d'}$ and $n \in \N$, we conclude with the help of the preceding lemma that $A$
induces a densely defined linear map
\begin{equation*}
  \F_J \supset \D_J \to \F_I, \quad f \mapsto f \circ A^*.
\end{equation*}
The crucial problem is to determine when this map is bounded. If $J$ is the vanishing ideal
of a single subspace $V \subset \C^{d'}$ and $A$ is isometric on $V$, then the map is in fact isometric.
This follows from results in \cite{davramshal}.
For the convenience of the reader, a proof is provided below.

\begin{lem}
  \label{lem:comp_single_subspace}
  Let $V \subset \C^{d'}$ be a subspace and let $J \subset \polyring{d'}$ be its vanishing ideal.
  If $A: \C^{d'} \to \C^d$ is a linear map which is isometric on $V$, then
  \begin{equation*}
    \compmap{A}: \F_J \supset \D_J \to H^2_d, \quad f \mapsto f \circ A^*
  \end{equation*}
  is an isometry.
\end{lem}

\begin{proof}
  Let $\lambda,\mu \in V \cap \B_{d'}$ and $k,n \in \N$ be arbitrary.
  Using the homogeneous decomposition of the reproducing kernel $K$
  of $H^2_d$, we see that
  \begin{align*}
    \big\langle \compmap{A} (\langle \cdot, \lambda \rangle^n) ,
      \compmap{A} (\langle \cdot, \mu \rangle^k) \big\rangle_{H^2_d}
    &=
    \big\langle \langle \cdot, A \lambda \rangle^n ,
      \langle \cdot, A \mu \rangle^k \big\rangle_{H^2_d} \\
    &=
    \delta_{kn} \big\langle \langle \cdot, A \lambda \rangle^n ,
      K(\cdot,A \mu) \big\rangle_{H^2_d} \\
    &= \delta_{kn} \, \langle A \mu, A \lambda \rangle^n
    = \delta_{kn} \, \langle \mu, \lambda \rangle^n.
  \end{align*}
  Similarly,
  \begin{equation*}
    \big\langle \langle \cdot, \lambda \rangle^n ,
      \langle \cdot, \mu \rangle^k \big\rangle_{H^2_{d'}}
      = \delta_{kn} \, \langle \mu, \lambda \rangle^n.
  \end{equation*}
  Since $\D_J$ is linearly spanned by polynomials of the form $\langle \cdot,\lambda \rangle^n$
  with $\lambda \in V \cap \B_{d'}$ and $n \in \N$ by the preceding lemma, we conclude that $\compmap{A}$ is isometric.
\end{proof}

When considering more complicated algebraic sets such as unions of subspaces, one of course
wishes to decompose the sets into smaller pieces which are easier to deal with.
Algebraically, this corresponds
to writing an ideal as an intersection of larger ideals. On the level of the
spaces $\F_I$, we get the following result.

\begin{lem}
  \label{lem:ideal_dec_F}
  Let $J_1, \ldots,J_r \subset \polyrin$ be homogeneous ideals and let
  $J = J_1 \cap \ldots \cap J_r$. Then
  \begin{equation*}
      \overline{J} = \overline{J_1} \cap \ldots \cap \overline{J_r},
  \end{equation*}
  and
  \begin{equation*}
    \F_J = \overline{\F_{J_1} + \ldots + \F_{J_r}}.
  \end{equation*}
\end{lem}

\begin{proof}
  It suffices to prove the first claim, since the second will then follow
  by taking orthogonal complements.
  To this end, note that the inclusion
  $\overline{J} \subset \overline{J_1} \cap \ldots
  \cap \overline{J_r}$
  is trivial.
  Conversely, it is an easy consequence of the homogeneity of the $J_k$
  that, for any element $f \in \overline{J_k}$ with homogeneous expansion
  \begin{equation*}
    f = \sum_{n=0}^\infty f_n,
  \end{equation*}
  each $f_n$ is contained in $J_k$, from which the reverse inclusion readily follows.
\end{proof}

The question under which conditions the sum $\F_{J_1} + \ldots + \F_{J_r}$ in the preceding lemma
is itself closed will be of central importance.
In general, $\F_{J_1} + \F_{J_2}$ need not be closed for two radical homogeneous ideals $J_1$ and $J_2$,
see Example \ref{exa:sum_not_closed} below.
But thanks to the reduction to unions of subspaces in \cite{davramshal},
we only need to consider the case where the $J_k$ are
vanishing ideals of subspaces in $\C^d$.

To keep the statements of the following results reasonably short, we make an ad-hoc definition which
will only be used in this section.
\begin{defn}
  Let $J \subset \polyring{d}$ be a radical homogeneous ideal, and let $V(J) = W_1 \cup \ldots
  \cup W_r$ be the decomposition of $V(J)$ into irreducible components.
  Denote the vanishing ideal of $\spa{W_k}$ by $\widehat J_k$. We call $J$
  \emph{\admideal} if the algebraic sum $\F_{\widehat J_{1}} + \ldots + \F_{\widehat J_r}$
  is closed.
\end{defn}

\begin{prop}
  \label{prop:good_bounded_maps}
  Let $I$ and $J$ be radical homogeneous ideals in $\polyring{d}$ and $\polyring{d'}$, respectively.
  Suppose that there is a linear map $A: \C^{d'} \to \C^{d}$ that maps $Z(J)$ bijectively onto
  $Z(I)$. If $J$ is \admideal, then
  \begin{equation*}
    \F_J \supset \D_J \to \F_I, \quad f \mapsto f \circ A^*
  \end{equation*}
  is a bounded map.
\end{prop}

\begin{proof}
  Let $V(J) = W_1 \cup \ldots \cup W_r$ be the irreducible decomposition of $V(J)$,
  and let $\widehat J_k$ be the vanishing ideal of $\spa(W_k)$. Define
  \begin{equation*}
    S = \spa (W_1) \cup \ldots \cup \spa(W_r),
  \end{equation*}
  and denote the vanishing ideal of $S$ by $\widehat J$, so that $\widehat J = \widehat J_1 \cap
  \ldots \cap \widehat J_r$. Since $\D_{J} \subset \D_{\widehat J}$, it suffices to show that
  $f \mapsto f \circ A^*$ defines a bounded map on $\D_{\widehat J}$.
  By Lemma \ref{lem:ideal_dec_F}, we have
  \begin{equation*}
    \F_{\widehat J} = \overline{\F_{\widehat J_1} + \ldots + \F_{\widehat J_r}}.
  \end{equation*}
  By Lemma 7.5 and Proposition 7.6 in \cite{davramshal}, the linear map $A$ is isometric
  on $S$. Consequently, Lemma \ref{lem:comp_single_subspace} shows that
  $f \mapsto f \circ A^*$ defines an isometry on each $D_{\widehat J_k} \subset
  \F_{\widehat J_k}$.
  We will use the hypothesis that $J$ is \admideal\ in order to show that
  $f \mapsto f \circ A^*$ defines a bounded map on $\D_{\widehat J}$. To this end,
  we note that since $\F_{\widehat J_1} + \ldots + \F_{\widehat J_r}$ is closed,
  a standard application of the open mapping theorem yields a constant $C \ge 0$
  such that for any $f \in \F_{\widehat J}$, there are $f_k \in \F_{\widehat J_k}$
  with $f = f_1 + \ldots + f_r$ and
  \begin{equation*}
    ||f_1||^2 + \ldots + ||f_r||^2 \le C ||f||^2.
  \end{equation*}
  If $f$ is a homogeneous polynomial of degree $n$, we can choose the $f_k$
  to be homogeneous polynomials of degree $n$ as well.
  Consequently, if $f \in \D_{\widehat J}$, the $f_k$
  can be chosen from $\D_{\widehat J_k}$. With such a choice, we obtain for
  $f \in \D_{\widehat J}$ the (crude) estimate
  \begin{align*}
    ||f \circ A^*||^2 &= ||f_1 \circ A^* + \ldots + f_r \circ A^*||^2 \\
    &\le r^2 \max_{1 \le k \le r} ||f_k \circ A^*||^2 \\
    &= r^2 \max_{1 \le k \le r} ||f_k||^2
    \le C r^2  ||f||^2,
  \end{align*}
  where we have used that $f \mapsto f \circ A^*$ is an isometry on each $\D_{\widehat J_k}$.
\end{proof}

In the setting of the preceding proposition, let $\compmap{A}: \F_J \to \F_I$ be
the continuous extension of $f \mapsto f \circ A^{*}$ onto $\F_J$.
Taking the homogeneous expansion of the kernel functions $K(\cdot,\lambda)$ into account, we infer
from \eqref{eqn:com_A_skal_prod} that
$\compmap{A}$ satisfies
\begin{equation*}
  \compmap{A}( K(\cdot,\lambda)) = K(\cdot,A \lambda) \quad \tfa \lambda \in Z^{0}(J).
\end{equation*}
The existence of topological isomorphisms between $\A_I$
and $\A_J$ if $I$ and $J$ are \admideal\ now follows exactly as in the proof
of \cite[Theorem 7.17]{davramshal}.

\begin{cor}
  \label{cor:good_algebra_iso}
  Let $I$ and $J$ be radical homogeneous ideals in $\polyring{d}$ and $\polyring{d'}$, respectively.
  Suppose that there are linear maps $A: \C^{d'} \to \C^d$ and $B: \C^d \to \C^{d'}$ which
  restrict to mutually inverse bijections $A: Z(J) \to Z(I)$ and $B: Z(I) \to Z(J)$.
  If $I$ and $J$ are \admideal,
  then $\compmap{A}$ and $\compmap{B}$ are inverse to each other, and
  \begin{equation*}
    \Phi: \A_I \to \A_J, \quad T \mapsto (\compmap{A})^* T  (\compmap{B})^*,
  \end{equation*}
  is a completely bounded isomorphism. Regarding $\A_I$ and $\A_J$ as function
  algebras on $Z(I)$ and $Z(J)$, respectively, $\Phi$ is given by composition with $A$, that is,
  \begin{equation*}
    \Phi (\varphi) = \varphi \circ A \quad \tfa \varphi \in \A_I.
    \eqno\qed
  \end{equation*}
\end{cor}

To improve the corresponding results from \cite{davramshal}, we will show that
every radical homogeneous ideal $I \subset \polyring{d}$ is automatically \admideal.
To this end, we will work with
the description of Drury-Arveson space as symmetric Fock space, rather than as a
Hilbert function space. We begin by recalling some standard definitions.

For a finite dimensional Hilbert space $E$, let
\begin{equation*}
  \fock{E} = \bigoplus_{n=0}^\infty E^{\otimes n}
\end{equation*}
be the full Fock space over $E$.
Note that if $V \subset E$ is a subspace, we can regard $\fock{V}$ as a subspace of
$\fock{E}$, and the orthogonal projection from $\fock{E}$ onto $\fock{V}$ is given by
\begin{equation*}
  P_{\fock{V}} = \bigoplus_{n=0}^\infty (P_V)^{\otimes n}.
\end{equation*}
Let $E^n \subset E^{\otimes n}$ denote
the $n$-fold symmetric tensor power of $E$, and write
\begin{equation*}
  \symfock{E} = \bigoplus_{n=0}^\infty E^{n} \subset \fock{E}
\end{equation*}
for the symmetric Fock space over $E$.
Then $H^2_d$ can be identified with $\symfock{\C^d}$
via an anti-unitary map $U: H^2_d \to \symfock{\C^d}$, which is uniquely determined by
\begin{equation}
  \label{eqn:DA_Fock}
  U (\langle \cdot,\lambda \rangle^{n}) = \lambda^{\otimes n}
\end{equation}
for all $\lambda \in \C^d$ and $n \in \N$ (see \cite[Section 1]{arv-s3}).

This identification allows us to translate the condition that the ideals $I$ and $J$ be \admideal\
in terms of symmetric Fock space. In fact, working with full Fock space suffices.

\begin{lem}
  \label{lem:full_Fock_good}
  Let $J \subset \polyring{d}$ be a radical homogeneous ideal, and let
  \begin{equation*}
    V(J) = W_1 \cup \ldots \cup W_r
  \end{equation*}
  be the irreducible decomposition of $V(J)$. Let $V_k = \spa{W_k}$.
  If the algebraic sum of the full Fock spaces $\fock{V_1} + \ldots + \fock{V_r}$ is closed,
  then $J$ is \admideal.
\end{lem}

\begin{proof}
  Let $\widehat J_k$ be the vanishing ideal of $V_k$.
  Then by Lemma \ref{lem:F_I_generating_set}, the linear span
  of the elements $\langle \cdot,\lambda \rangle^n$ with $\lambda \in V_k$
  and $n \in \N$ is dense in $\F_{\widehat J_k}$, whereas
  $\symfock{V_k}$ is the closed linear span of the symmetric tensors $\lambda^{\otimes n}$ with
  $\lambda \in V_k$ and $n \in \N$. Hence, the identity \eqref{eqn:DA_Fock} shows that
  $U$ maps $\F_{\widehat J_k}$ onto $\symfock{V_k}$, so that $J$ is \admideal\
  if and only if the algebraic sum
  \begin{equation*}
    S = \symfock{V_1} + \ldots + \symfock{V_r}
  \end{equation*}
  is closed.

  Now, let $Q$
  be the orthogonal
  projection from $\fock{\C^d}$ onto $\symfock{\C^d}$. It is well known that in degree $n$, we have
  \begin{equation*}
   Q \big|_{(\C^d)^{\otimes n}} = \frac{1}{n!} \sum_{\sigma \in S_n} U_\sigma,
  \end{equation*}
  where $S_n$ denotes the symmetric group on $n$ letters, and for $\sigma \in S_n$, the unitary
  operator $U_\sigma$ is given by
  \begin{equation*}
    U_\sigma (x_1 \otimes \ldots \otimes x_n) = x_{\sigma^{-1} (1)} \otimes \ldots \otimes x_{\sigma^{-1}(n)}.
  \end{equation*}
  Note that for a subspace $V \subset \C^d$, the projections $Q$ and $P_{\fock{V}}$ commute and
  $Q P_{\fock{V}} = P_{\symfock{V}}$,
  from which it easily follows that closedness of $\fock{V_1} + \ldots + \fock{V_r}$
  implies closedness of $S$. Indeed, if
  $x$ is in the closure of
  $S$, then
  we can write $x = \widetilde x_1 + \ldots + \widetilde x_r$ with $ \widetilde x_k \in \fock{V_k}$.
  Setting $x_k = Q \widetilde x_k \in \symfock{V_k}$, we have
  \begin{equation*}
    x = Q x =  x_1 + \ldots + x_r \in S. \qedhere
  \end{equation*}
\end{proof}

\section{The Friedrichs angle}
In order to show that sums of full Fock spaces are closed, we will make use of a classical
notion of angle between two closed subspaces of a Hilbert space due to Friedrichs \cite{friedrichs} (for the
history of this and related quantities, see for example \cite{boettcher}).
\begin{defn}
  Let $H$ be a Hilbert space and let $M,N \subset H$ be closed subspaces.
  If $M \not \subset N$ and $N \not \subset M$,
  the Friedrichs angle between $M$ and $N$ is defined to be the angle in $[0,\frac{\pi}{2}]$
  whose cosine is
  \begin{equation*}
  \fangle(M,N) = \sup_{\substack{x \in M \ominus (M \cap N)
  \\ y \in N \ominus (M \cap N) \\ x \neq 0 \neq y}} \frac{ | \langle x,y \rangle|}
  {||x|| \, ||y||}.
  \end{equation*}
  Otherwise, we set $\fangle(M,N) = 0$.
\end{defn}

We record some standard properties of the Friedrichs angle in the following lemma. For a closed subspace
$M$ of a Hilbert space $H$, we denote the orthogonal projection from $H$ onto $M$ by $P_M$.
\begin{lem}
  \label{lem:angle_standard_prop}
  Let $H$ be a Hilbert space and let $M$ and $N$ be closed subspaces of $H$.
  \begin{enumerate}[label=\normalfont{(\alph*)},ref={\thelem~(\alph*)}]
  \item $\fangle(M,N) = \fangle(M \ominus (M \cap N), N \ominus (M \cap N))$.
  \label{it:angle_standard_prop_disjoint}
  \item $\fangle(M,N) = ||P_M P_N - P_{M \cap N}||$ and $\fangle(M,N)^2 = ||P_N P_M P_N - P_{M \cap N}||$.
  \label{it:angle_standard_prop_square}
  \item $M+N$ is closed if and only if $c(M,N)< 1$.
  \label{it:angle_standard_prop_sum_closed}
  \end{enumerate}
\end{lem}

\begin{proof}
  (a) is obvious, and the first half of (b) and (c) are well known, see for example
  Lemma 10 and Theorem 13 in \cite{deutsch95}. To show the second half of (b), we set $T=P_M P_N - P_{M \cap N}$
  and note that
  \begin{equation*}
    T^* T = (P_N P_M - P_{M \cap N}) (P_M P_N - P_{M \cap N}) = P_N P_M P_N - P_{M \cap N}.
  \end{equation*}
  Hence, by the first half of (b),
  \begin{equation*}
    c(M,N)^2 = ||T||^2 = ||T^* T|| = ||P_N P_M P_N - P_{M \cap N}||.
    \qedhere
  \end{equation*}
\end{proof}

Part (c) is the reason why we are considering the Friedrichs angle. Recently,
Badea, Grivaux and M\"uller \cite{BGV} have introduced a generalization of the Friedrichs angle
to more than two subspaces.
Although we want to show closedness of sums of arbitrarily many Fock spaces,
an inductive argument using the classical definition for two subspaces
seems to be more feasible in our case.

As a first application, we exhibit two radical homogeneous ideals $I,J \subset \polyrin$
such that $\F_I + \F_J$ is not closed.
When the ideals are not necessarily radical, an example for this phenomenon is also given by Shalit's example
of a set of polynomials which
is not a stable generating set, see \cite[Example 2.6]{shalit}.

\begin{exa}
  \label{exa:sum_not_closed}
  Let $I = \langle y^2 + x z \rangle$ and $J=\langle x \rangle$ in $\C[x,y,z]$.
  We claim that $\F_I + \F_J$ is not closed.
  Since for two closed subspaces $M$ and $N$ of a Hilbert space $H$, closedness of $M+N$ is equivalent
  to closedness of $M^\bot + N^\bot$ (see for example \cite[Theorem 13]{deutsch95}), it suffices to show that
  $\overline{I} + \overline{J}$ is not closed.
  To this end, we set for $n \ge 2$
  \begin{equation*}
    f_n = z^{n-2} (y^2+xz) \quad \tand \quad g_n = z^{n-1} x.
  \end{equation*}
  Clearly, $f_n \in I$ and $g_n \in J$ for all $n$. Using that different monomials in $H^2_d$ are orthogonal,
  one easily checks that all $f_n$ and $g_n$ are orthogonal to $I \cap J = \langle x^2 z + x y^2 \rangle$,
  so they are orthogonal to $\overline{I} \cap \overline{J} = \overline{I \cap J}$ (see Lemma \ref{lem:ideal_dec_F})
  as well.
  Moreover, a straightforward calculations yields
  \begin{equation*}
    ||f_n||^2 = \frac{n+1}{n (n-1)} \quad \tand \quad \langle f_n,g_n \rangle = ||g_n||^2 = \frac{1}{n}.
  \end{equation*}
  Consequently,
  \begin{equation*}
    \frac{\langle f_n,g_n \rangle}{||f_n|| \, ||g_n||} = \sqrt{\frac{n-1}{n+1}} \xrightarrow{n \to \infty} 1,
  \end{equation*}
  from which we conclude that $\fangle(\overline{I},
  \overline{J})=1$, so that $\overline{I} + \overline{J}$
  is not closed by Lemma \ref{it:angle_standard_prop_sum_closed}.
\end{exa}

Let $H$ be a Hilbert space which is graded in the sense that $H$ is the orthogonal
direct sum
\begin{equation*}
  H = \bigoplus_{n \in \N} H_n
\end{equation*}
for some Hilbert spaces $H_n$.
Denote the orthogonal projection from $H$ to $H_n$ by $P_n$. We say that
a closed subspace $M \subset H$ is graded if $P_n P_M = P_M P_n$ for all $n \in \N$.
Equivalently,
\begin{equation*}
  M = \bigoplus_{n=0}^\infty M \cap H_n.
\end{equation*}
Note that $M$ is graded if and only if $P_M$ belongs to the commutant of $\{P_n: n \in \N\}$,
which is a von Neumann algebra. In particular, if $M,N \subset H$ are graded,
then $\overline{M+N}$ and $M \cap N$
are graded as well.
The most important examples of graded Hilbert spaces in our case are full Fock spaces and sums thereof.

The angle between two graded subspaces can be easily expressed in terms of the angles
between their graded components by the following formula.
\begin{lem}
  \label{lem:angle_graded_hilb_space}
  Let $H=\bigoplus_{n=0}^\infty H_n$ be a graded Hilbert space and let $M,N \subset H$ be graded subspaces.
  Write $M_n = M \cap H_n$ and $N_n = N \cap H_n$ for $n \in \N$.
  Then
  \begin{equation*}
    \fangle(M,N) = \sup_{n \in \N} \fangle(M_n,N_n).
  \end{equation*}
\end{lem}

\begin{proof}
  The assertion readily follows from Lemma
  \ref{it:angle_standard_prop_square} and the fact that
  for any graded subspace $K \subset H$, we have
  \begin{equation*}
    P_{K} = \bigoplus_{n=0}^\infty P_{K \cap H_n}^{H_n},
  \end{equation*}
  where $P_{K \cap H_n}^{H_n}$ denotes the orthogonal projection from $H_n$ onto $K \cap H_n$.
\end{proof}

If each of the spaces $H_n$ in the preceding lemma is finite dimensional, then
$\fangle(M_n,N_n) < 1$ for all $n \in \N$. This can easily be seen
from the definition of the Friedrichs angle, or, alternatively, it follows as an application
of Lemma \ref{it:angle_standard_prop_sum_closed}.
In particular, $M+N$ is closed
if and only if $\limsup_{n \to \infty} \fangle(M_n,N_n) < 1$.
That is, closedness of $M+N$ only depends on the asymptotic behaviour of the sequence
$(\fangle(M_n,N_n))_n$.
Inspired by condition 7 in \cite[Theorem 2.3]{BGV},
we will now introduce a variant of the Friedrichs angle which
reflects this fact.
For a closed subspace $M$ of a Hilbert space $H$, we denote the equivalence class of $P_M$ in
the Calkin algebra by $p_M$.

\begin{defn}
  Let $H$ be a Hilbert space and let $M,N \subset H$ be closed subspaces.
  The essential
  Friedrichs angle is defined to be the angle in $[0,\frac{\pi}{2}]$ whose cosine is
  \begin{equation*}
    \fangle_e(M,N) = ||p_M p_N - p_{M \cap N}||.
  \end{equation*}
\end{defn}

Parts (a) and (b) of Lemma \ref{lem:angle_standard_prop} also hold with $\fangle_e$ in place
of $\fangle$.
\begin{lem}
  \label{lem:ess_angle_standard_prop}
  Let $H$ be a Hilbert space and let $M,N \subset H$ be closed subspaces.
  \begin{enumerate}[label=\normalfont{(\alph*)},ref={\thelem~(\alph*)}]
    \item $\fangle_e(M,N) = \fangle_e(M \ominus (M \cap N), N \ominus (M \cap N))$.
    \label{it:ess_angle_standard_prop_disjoint}
    \item $\fangle_e(M,N)^2 = ||p_N p_M p_N - p_{M \cap N}||$.
    \label{it:ess_angle_standard_prop_square}
  \end{enumerate}
\end{lem}

\begin{proof}
  (a) follows from the identity
\begin{equation*}
  (P_M - P_{M \cap N}) ( P_N - P_{M \cap N}) = P_M P_N - P_{M \cap N},
\end{equation*}
while (b) is again an application of the $C^*$-identity, see the proof of Lemma \ref{lem:angle_standard_prop}.
\end{proof}

To determine if $M+N$ is closed, the essential Friedrichs angle
is just as good as the usual one, that is, part (c) of Lemma \ref{lem:angle_standard_prop}
holds with $\fangle_e$ in place of $\fangle$ as well. This follows from \cite[Theorem 2.3]{BGV}.
For the convenience of the reader, a short proof is provided below.
First, we record a simple lemma.

\begin{lem}
  \label{lem:proj_intersection_point_spectrum}
  Let $H$ be a Hilbert space and let $M_1 ,\ldots , M_r \subset H$
  be closed subspaces. Define $T = P_{M_1} P_{M_2} \ldots P_{M_r}$ and $M = M_1 \cap \ldots \cap M_r$.
  \begin{enumerate}[label=\normalfont{(\alph*)},ref={\thelem~(\alph*)}]
  \item $\ker(1-T^* T) = M$.
  \label{it:proj_intersection_point_spectrum_kernel}
  \item If $\dim H < \infty$, then $||T|| = 1$ if and only if $M \neq \{0\}$.
  \label{it:proj_intersection_point_spectrum_fin_dim}
  \end{enumerate}
\end{lem}

\begin{proof}
  We first claim that a vector $x \in H$ satisfies $||T x|| = ||x||$ if and only if $x \in M$.
  We prove the non-trivial implication
  by induction on $r$. The case $r=1$ is clear.
  So suppose that $r \ge 2$ and that the assertion is true for $r-1$ subspaces.
  Let $x \in H$ such that $||T x|| = ||x||$. Setting $y = P_{M_2} \ldots P_{M_r} x$, we have
  \begin{equation*}
    ||x|| = ||P_{M_1} y|| \le ||y|| \le ||x||,
  \end{equation*}
  hence $y \in M_1$ and $||P_{M_2} \ldots P_{M_r} x|| = ||x||$.
  The inductive hypothesis implies that
  $x \in M_2 \cap \ldots \cap M_r$, and thus also $x = y \in M_1$, which finishes the proof of the claim.

  Both assertions easily follow from this observation. Clearly, $M$ is contained in $\ker(1-T^* T)$. Conversely,
  any $x \in \ker(1-T^* T)$ satisfies $||x||^2 = ||T x||^2$, so that $x \in M$ by the above remark, which proves (a).

  Part (b) is immediate from the claim as well, since $||T||$ is attained if $H$ is finite dimensional.
\end{proof}

\begin{lem}
  \label{lem:sum_closed_ess_angle}
  Let $H$ be a Hilbert space and let $M,N \subset H$ be closed subspaces.
  Then $M+N$ is closed if and only if $\fangle_e(M,N) < 1$.
\end{lem}

\begin{proof}
  In view of Lemma \ref{it:angle_standard_prop_sum_closed}, it is sufficient to show that $\fangle(M,N) < 1$ if
  $\fangle_e(M,N) < 1$, since $\fangle_e(M,N) \le \fangle(M,N)$ holds trivially.
  To this end, we can assume without loss of generality that $M \cap N = \{0\}$ by
  Lemma \ref{it:angle_standard_prop_disjoint} and Lemma \ref{it:ess_angle_standard_prop_disjoint}.
  Then $||P_N P_M P_N||_e < 1$, so $T=1-P_N P_M P_N$ is a self-adjoint Fredholm operator.
  Lemma \ref{it:proj_intersection_point_spectrum_kernel} implies that $T$ is injective, from
  which we conclude that $T$ is invertible. It follows that $1 \not \in \sigma(P_N P_M P_N)$,
  and hence that $\fangle(M,N) = ||P_N P_M P_N|| < 1$.
\end{proof}

For graded subspaces, we obtain a more concrete description of the essential Friedrichs angle, which
gives another proof for the preceding lemma in the graded case. In particular, we see that the
essential Friedrichs angle indeed only depends on the asymptotic behaviour of the Friedrichs
angles between the graded components.

\begin{lem}
  \label{lem:ess_angle_limsup}
  Let $H = \bigoplus_{n=0}^\infty H_n$ be a graded Hilbert space, where all $H_n$
  are finite dimensional, and let $M,N \subset H$ be graded subspaces. Write $M_n = M \cap H_n$ and $N_n = N \cap H_n$
  for $n \in \N$. Then
  \begin{equation*}
    \fangle_e(M,N) = \limsup_{n \to \infty} \fangle(M_n, N_n).
  \end{equation*}
\end{lem}

\begin{proof}
  Let $\varepsilon > 0 $ be arbitrary. By definition of $\fangle_e$, there is a compact
  operator $K$ on $H$ such that
  \begin{equation*}
    ||P_M P_N - P_{M \cap N} + K|| \le \fangle_e(M,N) + \varepsilon.
  \end{equation*}
  It is easy to see that $\lim_{n \to \infty} ||P_n K P_n|| = 0$.
  Furthermore,
  \begin{align*}
    \fangle(M_n, N_n) &=
    ||P_n (P_M P_N - P_{M \cap N}) P_n|| \\
    &\le ||P_n (P_M P_N - P_{M \cap N} + K) P_n|| + ||P_n K P_n|| \\
    &\le \fangle_e(M,N) + \varepsilon + ||P_n K P_n||,
  \end{align*}
  so $\limsup_{n \to \infty} \fangle(M_n,N_n) \le \fangle_e(M,N)$.

  Conversely, for any $k \in \N$, the operator
  \begin{equation*}
    K= \bigoplus_{n=0}^k P_n (P_M P_N - P_{M \cap N}) P_n
  \end{equation*}
  has finite rank, and
  \begin{equation*}
    P_M P_N - P_{M \cap N} -K = \bigoplus_{n=k+1}^\infty
    P_n(P_M P_N -P_{M \cap N}) P_n.
  \end{equation*}
  Hence
  \begin{equation*}
    \fangle_e(M,N) \le ||P_M P_N - P_{N \cap N} -K|| = \sup_{n \ge k+1} \fangle(M_n,N_n)
  \end{equation*}
  for all natural numbers $k$,
  which establishes the reverse inequality.
\end{proof}

\begin{rem*}
  If $T$ is an operator on a Hilbert space $H$, the infimum
  \begin{equation*}
    \inf \{ ||T + K||: K \in \K(H) \}
  \end{equation*}
  is always attained \cite{holmes}.
  In particular,
  we can choose an operator $K$ in the first part of the above proof
  such that $||P_M P_N -P_{M \cap N} + K|| = \fangle_e(M,N)$.
\end{rem*}

\section{Reduction to subspaces with trivial joint intersection}

Let $V_1,\ldots,V_r$ be subspaces of $\C^d$.
In this section, we will reduce the problem of showing closedness of the sum of Fock
spaces
$\F(V_1) + \ldots + \F(V_r) \subset \fock{\C^d}$ to the case where $V_1 \cap \ldots \cap V_r = \{0\}$.
Note that in \cite[Lemma 7.12]{davramshal}, Davidson, Ramsey and Shalit reduced the problem of showing
boundedness of the map $f \mapsto f \circ A^*$ in the setting of unions of subspaces
to the case where the joint intersection of the subspaces is trivial.
However, in our situation, it does not suffice to consider only subspaces
with trivial joint intersection.
The issue is that in the inductive proof of closedness of the sum of $r$ Fock spaces,
we will use the inductive hypothesis on $r-1$ subspaces which do not necessarily
have trivial joint intersection.

We begin with two simple consequences of the Gelfand-Naimark theorem.

\begin{lem}
  \label{lem:C_alg_gelfand_consequences}
  Let $\A$ be a unital $C^*$-algebra and let $a,b \in \A$ be self-adjoint elements.
\begin{enumerate}[label=\normalfont{(\alph*)},ref={\thelem~(\alph*)}]
  \item If $a b= 0$, then $||a+b|| = \max(||a||,||b||)$.
  \label{it:C_alg_gelfand_consequences_prod_zero}
  \item Suppose that $a$ and $b$ commute and that $a \le b$.
  If $f$ is a continuous and increasing real-valued function on $\sigma(a) \cup \sigma(b)$, then $f(a) \le f(b)$.
  \label{it:C_alg_gelfand_consequences_increasing}
\end{enumerate}
\end{lem}

\begin{proof}
  In both cases, the unital $C^*$-algebra generated by $a$ and $b$ is commutative. By the Gelfand-Naimark
  theorem, we can therefore regard $a$ and $b$ as real-valued functions on a compact Hausdorff space, where
  both assertions are elementary.
\end{proof}

\begin{lem}
\label{lem:angle_4_sum_orth}
Let $H$ be a Hilbert space and let $M_1,M_2,N_1,N_2 \subset H$ be closed subspaces with
$M_1 \bot M_2, M_1 \bot N_2, M_2 \bot N_1, N_1 \bot N_2$. Then
\begin{equation*}
\fangle(M_1 \oplus M_2, N_1 \oplus N_2) = \max( \fangle(M_1,N_1),\fangle(M_2,N_2)).
\end{equation*}
The same is true with $\fangle_e$ in place of $\fangle$.
\end{lem}

\begin{proof}
The assertion can be shown using the definition of the Friedrichs angle
or working with projections. The latter has the advantage of proving the
claim for the essential Friedrichs angle at the same time.

First, we note that the assumptions on the subspaces imply that
\begin{equation*}
(M_1 \oplus M_2) \cap (N_1 \oplus N_2) = (M_1 \cap N_1) \oplus (M_2 \cap N_2).
\end{equation*}
Indeed, if $m_1 + m_2 = n_1 + n_2$ is an element of the space on the left-hand side,
with $m_i \in M_i$ and $n_i \in N_i$ for $i=1,2$, then
$m_1 - n_1 = n_2 - m_2$, and the orthogonality relations show that this vector is zero.
Hence $m_1 \in M_1 \cap N_1$ and $m_2 \in M_2 \cap N_2$, thus proving the non-trivial
inclusion. Using the orthogonality relations once again, we conclude that
\begin{align*}
&P_{N_1 \oplus N_2} P_{M_1 \oplus M_2} P_{N_1 \oplus N_2} - P_{ (M_1 \oplus M_2) \cap (N_1 \oplus N_2)} \\
= (&P_{N_1}+ P_{N_2}) (P_{M_1} + P_{M_2})(P_{N_1} + P_{N_2}) - (P_{M_1 \cap N_1} + P_{M_2 \cap N_2}) \\
= (&P_{N_1} P_{M_1} P_{N_1} - P_{M_1 \cap N_1}) + ( P_{N_2} P_{M_2} P_{N_2} - P_{M_2 \cap N_2}).
\end{align*}
Since
\begin{equation*}
(P_{N_1} P_{M_1} P_{N_1} - P_{M_1 \cap N_1}) ( P_{N_2} P_{M_2} P_{N_2} - P_{M_2 \cap N_2}) = 0,
\end{equation*}
both assertions follow from Lemma \ref{it:C_alg_gelfand_consequences_prod_zero}.
\end{proof}

Tensoring with another Hilbert space does not make the angle worse.
\begin{lem}
\label{lem:angle_tensor_same_space}
Let $H$ be a Hilbert space and let $M,N \subset H$  be closed subspaces.
If $E$ is another non-trivial Hilbert space, then
\begin{equation*}
\fangle(M \otimes E, N \otimes E) = \fangle(M,N).
\end{equation*}
\end{lem}

\begin{proof}
First, note that $(M \cap N) \otimes E = (M \otimes E) \cap (N \otimes E)$. Since
$P_{K \otimes E} = P_K \otimes P_E$ for any closed subspace $K \subset H$, we have
\begin{align*}
||P_{M \otimes E} P_{N \otimes E} - P_{(M \otimes E) \cap (N \otimes E)}||
&=
||P_{M \otimes E} P_{N \otimes E} - P_{(M  \cap N)  \otimes E}|| \\
&= || (P_M P_N - P_{M \cap N} ) \otimes 1_E|| \\
&= ||P_M P_N - P_{ M \cap N}||. \qedhere
\end{align*}
\end{proof}

We can now prove the main result of this section.
It enables the desired reduction to subspaces with trivial joint intersection.
\begin{lem}
  \label{lem:angle_several_sum_perp}
  \label{cor:Fock_closed_reduction}
  Let $V_1,\ldots,V_r \subset \C^d$ be subspaces and let $V= V_1 \cap \ldots \cap V_r \neq \{0\}$.
  Suppose that $\F(V_1) + \ldots + \F(V_{r-1})$ and $\F(V_1 \ominus V) + \ldots + \F(V_{r-1} \ominus V)$
  are closed. Then $\F(V_1) + \ldots + \F (V_r)$
  is closed if and only if $\F(V_1 \ominus V) + \ldots + \F(V_r \ominus V)$ is closed.
\end{lem}

\begin{proof}
  We claim that it suffices
  to prove the following assertion:
  If $W_1, \ldots, W_r \subset \C^d$ are subspaces,
  and if $E \subset \C^d$ is a non-trivial subspace that is orthogonal
  to each $W_i$, then
  \begin{equation}\begin{split}
    \label{eqn:sum_closed_red_formula}
    &\fangle( (W_1 \oplus E)^{\otimes n} + \ldots +
    (W_{r-1} \oplus E)^{\otimes n} , (W_r \oplus E)^{\otimes n}) \\
    = \quad &\max_{j=1,\ldots,n}
    \fangle( W_1^{\otimes j} +
    \ldots + W_{r-1}^{\otimes j}, W_r^{\otimes j} ).
  \end{split}\end{equation}
  Indeed, setting $E=V$ and $W_i = V_i \ominus V$ for each $i$, we see from
  Lemma \ref{lem:angle_graded_hilb_space} and
  Lemma \ref{it:angle_standard_prop_sum_closed}
  that this assertion will prove the lemma.

  In fact, we will show that
  \begin{equation}
    \label{eqn:sum_closed_red_proof}
    \begin{split}
    &\fangle \Big( \sum_{i=1}^{r-1} W_i^{\otimes k} \otimes (W_i \oplus E)^{\otimes n}
    , W_r^{\otimes k} \otimes (W_r \oplus E)^{\otimes n} \Big) \\
    = \quad &\max_{j=k,\ldots,k+n}
    \fangle \Big( \sum_{i=1}^{r-1} W_i^{\otimes j}, W_r^{\otimes j} \Big)
    \end{split}
  \end{equation}
  holds for all natural numbers $k$ and $n$. The assertion \eqref{eqn:sum_closed_red_formula} corresponds
  to the case $k=0$, with the usual convention $W^{\otimes 0} = \C$ for a subspace $W \subset \C^d$.
  We proceed by induction on $n$. If $n=0$, this is trivial. So suppose that $n \ge 1$
  and that the assertion has been proved for $n-1$. First, we note
  that
  \begin{align*}
    &W_i^{\otimes k} \otimes (W_i \oplus E)^{\otimes n} \\
    = \quad &\big( W_i^{\otimes k+1} \otimes (W_i \oplus E)^{\otimes n-1} \big)
    \oplus \big( W_i^{\otimes k} \otimes E \otimes (W_i \oplus E)^{\otimes n-1}
    \big),
  \end{align*}
  holds for all $i$. So defining
  \begin{align*}
    M_1 &= \sum_{i=1}^{r-1} W_i^{\otimes k+1} \otimes (W_i \oplus E)^{\otimes n-1} \quad \tand \\
    M_2 &= \sum_{i=1}^{r-1} W_i^{\otimes k} \otimes E \otimes (W_i \oplus E)^{\otimes n-1},
  \end{align*}
  as well as
  \begin{align*}
    N_1 &= W_r^{\otimes k+1} \otimes (W_r \oplus E)^{\otimes n-1} \quad \tand \\
    N_2 &= W_r^{\otimes k} \otimes E \otimes (W_r \oplus E)^{\otimes n-1},
  \end{align*}
  we have
  \begin{align*}
    \sum_{i=1}^{r-1} W_i^{\otimes k} \otimes (W_i \oplus E)^{\otimes n} &= M_1 + M_2 \quad \tand \\
    W_r^{\otimes k} \otimes (W_r \oplus E)^{\otimes n} &= N_1 + N_2.
  \end{align*}
  Since $E$ is orthogonal to each $W_i$, we see that $M_1 \bot M_2, M_1 \bot N_2, M_2 \bot N_1$
  and $N_1 \bot N_2$.
  Consequently, Lemma \ref{lem:angle_4_sum_orth} applies to show that
  the left-hand side of \eqref{eqn:sum_closed_red_proof} equals
  \begin{equation*}
    \max( \fangle(M_1,N_1), \fangle(M_2,N_2)).
  \end{equation*}
  By induction hypothesis,
  \begin{equation*}
    \fangle(M_1,N_1) = \max_{j=k+1, \ldots , k+n}
    \fangle \Big( \sum_{i=1}^{r-1} W_i^{\otimes j}, W_r^{\otimes j} \Big).
  \end{equation*}
  Moreover, an application of Lemma \ref{lem:angle_tensor_same_space} combined with
  the inductive hypothesis shows that
  \begin{equation*}
    \fangle(M_2,N_2) = \max_{j=k, \ldots, k+n-1}
    \fangle \Big( \sum_{i=1}^{r-1} W_i^{\otimes j}, W_r^{\otimes j} \Big),
  \end{equation*}
  which finishes the proof.
\end{proof}

\begin{exa}
  With the formula derived in the proof of the preceding lemma, we can already determine
  the Friedrichs angle between two full Fock spaces.
  To begin with, suppose that $V_1$ and $V_2$
  are two subspaces in $\C^{d}$ such that $V_1 \cap V_2 = \{0\}$. Then
  Lemma \ref{it:angle_standard_prop_square}
  yields for all natural numbers $n$ the identity
  \begin{equation*}
    \fangle(V_1^{\otimes n}, V_2^{\otimes n}) = ||P_{V_1}^{\otimes n} P_{V_2}^{\otimes n}||
    = ||P_{V_1} P_{V_2}||^n = \fangle(V_1,V_2)^n.
  \end{equation*}
  Note that $\fangle(V_1,V_2) < 1$ because $\C^d$ is finite dimensional.
  If $V_1 \cap V_2 \neq \{0\}$, we set $W_i = V_i \ominus (V_1 \cap V_2)$ for $i=1,2$.
  By formula \eqref{eqn:sum_closed_red_formula},
  we have
  \begin{equation*}
    \fangle(V_1^{\otimes n}, V_2^{\otimes n}) = \max_{j=1,\ldots,n} \fangle(W_1^{\otimes j},
    W_2^{\otimes j})
  \end{equation*}
  for all $n$.
  Since $W_1$ and $W_2$ have trivial intersection,
  \begin{equation*}
    \fangle(W_1^{\otimes j}, W_2^{\otimes j}) = \fangle(W_1,W_2)^{j} = \fangle(V_1,V_2)^j
  \end{equation*}
  by what we have just proved, so
  \begin{equation*}
    \fangle(V_1^{\otimes n}, V_2^{\otimes n}) = \fangle(V_1,V_2)
  \end{equation*}
  for all $n$.
  As an application of Lemma \ref{lem:angle_graded_hilb_space}, we see that in any case,
  \begin{equation*}
    \fangle(\fock{V_1},\fock{V_2}) = \fangle(V_1,V_2),
  \end{equation*}
  while Lemma \ref{lem:ess_angle_limsup} shows that
  \begin{equation*}
    \fangle_e(\fock{V_1},\fock{V_2})
    = \begin{cases} \fangle(V_1,V_2), & \text{ if } V_1 \cap V_2 \neq \{0\}, \\
    0, & \text{ if } V_1 \cap V_2 = \{0\}.
    \end{cases}
  \end{equation*}
  In particular, we see that sums of \emph{two} Fock spaces are closed.
\end{exa}

We conclude this section with a lemma about the case of trivial joint intersection.
In view of the definition of the essential Friedrichs angle, it indicates why the reduction
to this case will be helpful.
\begin{lem}
  \label{lem:proj_product_compact}
  Let $V_1,\ldots, V_r \subset \C^d$ be subspaces with $V_1 \cap \ldots \cap V_r = \{0\}$.
  Set $M_i = \F(V_i)$ for $i=1,\ldots, r$. Then
  $P_{M_1} \ldots P_{M_r}$ is a compact operator.
\end{lem}

\begin{proof}
  We note that for each $i$,
  \begin{equation*}
    P_{M_i} = \bigoplus_{n=0}^\infty P_{V_i}^{\otimes n},
  \end{equation*}
  hence
  \begin{equation*}
    P_{M_1} \ldots P_{M_r} = \bigoplus_{n=0}^\infty (P_{V_1} \ldots P_{V_r})^{\otimes n}.
  \end{equation*}
  Since $V_1 \cap \ldots \cap V_r = \{0\}$, and since $\C^d$ is finite dimensional,
  $||P_{V_1} \ldots P_{V_r}|| < 1$ by Lemma \ref{it:proj_intersection_point_spectrum_fin_dim}.
  Therefore,
  \begin{equation*}
    ||(P_{V_1} \ldots P_{V_r})^{\otimes n}|| =
    ||(P_{V_1} \ldots P_{V_r})||^n \xrightarrow{n \to \infty} 0.
  \end{equation*}
  From this observation, it is easy to see that $P_{M_1} \ldots P_{M_r}$ is compact.
\end{proof}

\section{A closedness result}

In this section, we will deduce a closedness result which will form the inductive step in the proof
of our general result on the closedness of algebraic sums of $r$ Fock spaces.
Because of Lemma \ref{lem:angle_several_sum_perp} and Lemma
\ref{lem:proj_product_compact}, we will
consider the following situation throughout this section:
Let $r \ge 2$, and let $M_1,\ldots,M_r$
be closed subspaces of a Hilbert space $H$ which satisfy the following two conditions:
\begin{enumerate}[label=\normalfont{(\alph*)}]
\item Any algebraic sum of $r-1$ or fewer subspaces of the $M_i$ is closed, that is, for any subset
$\{i_1,\ldots,i_k\} \subset \{1,\ldots,r\}$ with $k \le r-1$,
the sum
\begin{equation*}
  M_{i_1} + \ldots + M_{i_k}
\end{equation*}
is closed.
\label{it:cond_a}
\item Any product of the $P_{M_i}$ containing each $P_{M_i}$ at least once is compact, that is,
  for any collection of (not necessarily distinct) indices $i_1,\ldots,i_k$ with $\{i_1,\ldots,i_k \}
  = \{1,\ldots,r\}$, the operator
  \begin{equation*}
    P_{M_{i_1}} P_{M_{i_2}} \ldots P_{M_{i_k}}
  \end{equation*}
  is compact.
  \label{it:cond_b}
\end{enumerate}
Our goal is to show that under these assumptions, the sum $M_1 + \ldots + M_r$ is closed.
Note that for $r=2$, the first condition is empty, while the second is equivalent to demanding that
$P_{M_1} P_{M_2}$ be compact.

Recall that for a closed subspace $M \subset H$, we denote the equivalence class of $P_M$ in the
Calkin algebra by $p_M$. Moreover, we define
$\A$ to be the unital $C^*$-subalgebra of the Calkin algebra
generated by $p_{M_1}, \ldots,p_{M_r}$.
The following proposition is the key step in proving that the sum $M_1 + \ldots + M_r$ is closed. It crucially
depends on condition \ref{it:cond_b}.

\begin{prop}
  \label{prop:special_representation}
  For any irreducible representation $\pi$ of $\A$ on a Hilbert space $K$,
  there is an $i \in \{1,\ldots,r\}$ such that $\pi(p_{M_i}) = 0$.

  In particular, there are representations $\pi_1, \ldots, \pi_r$ of $\A$ such that
  $\pi_i(p_{M_i}) = 0$ for each $i$, and such that $\pi = \bigoplus_{i=1}^r \pi_i$
  is a faithful representation of $\A$.
\end{prop}

\begin{proof}
  We write $p_i = p_{M_i}$.
  Suppose that $\pi(p_2), \ldots , \pi(p_r)$ are all non-zero. We have to prove that $\pi(p_1) = 0$.
  First, note that by condition \ref{it:cond_b},
  \begin{equation}
    \label{eqn:irred_rep}
    \pi(p_1 a_1 p_2 a_2 \ldots a_{r-1} p_r) = 0
  \end{equation}
  holds if each of the $a_i$ is a monomial in the $p_j$. By linearity and continuity,
  \eqref{eqn:irred_rep} therefore holds for all $a_1, \ldots, a_{r-1} \in  \A$.

  Since $\pi$ is irreducible, and since $\pi(p_r) \neq 0$, we have
  \begin{equation*}
    \bigvee_{a_{r-1} \in \A} \pi(a_{r-1} p_r) K = K.
  \end{equation*}
  Consequently, \eqref{eqn:irred_rep} implies that $\pi(p_1 a_1 p_2 a_2 \ldots a_{r-2} p_{r-1})=0$.
  Iterating this process yields the conclusion $\pi(p_1) = 0$, as desired.

   To establish the additional assertion, let $\pi_i$ be the direct sum of all irreducible GNS representations
   $\pi_f$ with $\pi_f(p_i)=0$, which is understood to be zero if there are no such representations.
   Then $\pi = \bigoplus_{i=1}^r \pi_i$ contains every irreducible GNS representation
   of $\A$ as a summand by the first part, and is therefore faithful.
\end{proof}

We will use the preceding proposition to get a good estimate of the essential Friedrichs angle
\begin{equation*}
  \fangle_e(M_1 + \ldots + M_{r-1}, M_r) =
  ||p_{M_1 + \ldots + M_{r-1}} p_{M_r} - p_{(M_1 + \ldots + M_{r-1}) \cap M_r}||.
\end{equation*}
To this end, we have to make sure that all occurring elements belong to $\A$.
Part of this is done by the following lemma.

\begin{lem}
  \label{lem:sum_closed_spectrum_and_algebra_membership}
  Let $H$ be a Hilbert space and let $M,N,N_1,\ldots,N_s \subset H$ be closed subspaces.
  \begin{enumerate}[label=\normalfont{(\alph*)},ref={\thelem~(\alph*)}]
  \item The algebraic sum $N_1 + \ldots + N_s$ is closed if and only if $0$ is not a cluster
  point of the spectrum of the positive operator $P_{N_1} + \ldots + P_{N_s}$. In this case, the image
  of the operator $P_{N_1} + \ldots + P_{N_r}$ equals $N_1 + \ldots + N_r$.
  \label{it:sum_closed_spectrum_and_algebra_membership_spectrum}
  \item If $N_1 + \ldots + N_s$ is closed, then
    \begin{equation*}
      P_{N_1 + \ldots + N_s} = \chi_{(0,\infty)} (P_{N_1} + \ldots + P_{N_s}),
    \end{equation*}
    where $\chi_{(0,\infty)}$ denotes the indicator function of $(0,\infty)$.
    In particular, the projection $P_{N_1 + \ldots + N_s}$
    belongs to the $C^*$-algebra generated by $P_{N_1}, \ldots, P_{N_s}$.
  \label{it:sum_closed_spectrum_and_algebra_membership_sum}
  \item $M+N$ is closed if and only if the sequence $((P_{M} P_{N} P_{M})^n)_n$ converges
  in norm to $P_{M \cap N}$. In particular, if $M+N$ is closed, then $P_{M \cap N}$ belongs to
  the $C^*$-algebra generated by $P_{M}$ and $P_{N}$.
  \label{it:sum_closed_spectrum_and_algebra_membership_intersection}
  \end{enumerate}
\end{lem}

\begin{proof}
  (a) Consider the continuous operator
  \begin{equation*}
    T: \bigoplus_{i=1}^s N_i \to H, \quad (x_i)_{i=1}^s \mapsto \sum_{i=1}^s x_i.
  \end{equation*}
  Clearly, the image of $T$ equals $N_1 + \ldots + N_r$.
  Consequently, this sum is closed if and only if the image of $T$ is closed, which, in turn,
  happens if and only if the image
  of $T^*$ is closed.
  It is easy to check that $T^*$ is given by
  $T^* x = (P_{N_1} x, \ldots, P_{N_s} x)$, so $T T^* = P_{N_1} + \ldots + P_{N_s}$.
  Hence the assertion follows from the general fact that the range of an operator $S$ is closed
  if and only if $0$ is not a cluster point of $\sigma(S^* S)$. The additional claim is now obvious.

  (b) Part (a) shows that the restriction of $\chi_{(0,\infty)}$ to $\sigma(P_{N_1} + \ldots + P_{N_s})$
  is continuous, so
  \begin{equation*}
    P = \chi_{(0,\infty)} (P_{N_1} + \ldots + P_{N_s})
  \end{equation*}
  belongs to the
$C^*$-algebra generated by $P_{N_1}, \ldots, P_{N_s}$. By standard properties of the functional calculus, $P$
is the orthogonal projection onto the range of $P_{N_1} + \ldots + P_{N_s}$, which
is $N_1 + \ldots + N_s$.

(c) For any $n \in \N$, we have
\begin{equation*}
||(P_M P_N P_M)^n - P_{M \cap N}|| = || (P_M P_N P_M - P_{M \cap N})^n|| = \fangle(M,N)^{2 n},
\end{equation*}
which converges to zero if and only if $\fangle(M,N) < 1$. This, in turn, is equivalent to $M+N$ being
closed by Lemma \ref{it:angle_standard_prop_sum_closed}.
\end{proof}

\begin{rem}
  \label{rem:alternating_projections}
  Statement (c) in the preceding lemma is just part of a bigger picture:
  For any closed subspaces $M,N \subset H$, the sequence $((P_M P_N)^n)_n$ (and hence
  also $((P_M P_N P_M)^n)_n = ((P_M P_N)^n P_M)_n$) converges in the strong
  operator topology to $P_{M \cap N}$, and the convergence is in norm if and only if $M+N$ is closed,
  see for example \cite[Section 3]{deutsch95}.
\end{rem}

Because of condition \ref{it:cond_a}, the preceding lemma
shows that $p_{M_1 + \ldots + M_{r-1}} \in \A$. If $r \ge 3$, we define for $i=1,\ldots,r-1$
\begin{equation*}
    S_i = M_1 + \ldots + \widehat{M_i} + \ldots + M_{r-1},
\end{equation*}
where $\widehat M_i$ stands for omission of $M_i$.
If $r=2$, this is understood to be the zero vector space.
Note that $S_i$ is a sum of $r-2$ subspaces for $r \ge 3$.
Thus another application of Lemma \ref{lem:sum_closed_spectrum_and_algebra_membership} shows that
$p_{S_i}$ and $p_{S_i \cap M_r}$ belong to $\A$.
However, care must be taken when using Proposition \ref{prop:special_representation} to estimate
$\fangle_e(M_1 + \ldots + M_{r-1}, M_r)$
since it is not obvious a priori that $p_{(M_1 + \ldots + M_{r-1}) \cap M_r}$
lies in $\A$.
Before we address this question, we record the following simple lemma for future reference.

\begin{lem}
  \label{lem:special_rep_proj_sum}
  Let $\pi= \bigoplus_{i=1}^r \pi_i$ be the representation from Proposition
  \ref{prop:special_representation}. Then for $i=1,\ldots,r-1$,
  \begin{equation*}
    \pi_i (p_{M_1 + \ldots + M_{r-1}}) = \pi_i (p_{S_i}).
  \end{equation*}
\end{lem}

\begin{proof}
  Let $a = p_{M_1} + \ldots + p_{M_{r-1}}$ and $b = p_{M_1} + \ldots + \widehat{p_{M_i}} + \ldots
  + p_{M_{r-1}}$ (if $r=2$, we set $b=0$).
  By condition \ref{it:cond_a} and Lemma \ref{lem:sum_closed_spectrum_and_algebra_membership},
  the origin is neither a cluster point of $\sigma(a)$ nor one of $\sigma(b)$, and
  \begin{equation*}
    \chi_{(0,\infty)} (a) = p_{M_1 + \ldots + M_{r-1}} \quad \tand \quad
    \chi_{(0,\infty)} (b) = p_{S_i}.
  \end{equation*}
  The assertion therefore follows from
  the identity $\pi_i(a) = \pi_i(b)$
  and the fact that the continuous functional calculus is compatible with $*$-homomorphisms.
\end{proof}

The question whether $p_{(M_1 + \ldots + M_{r-1}) \cap M_r}$
belongs to $\A$ is more difficult.
We will see below that it can well happen that for subspaces $M$ and $N$ of a Hilbert space $H$,
the projection $p_{M \cap N}$ does not belong to the unital $C^*$-algebra generated by $p_M$ and $p_N$.
Moreover, although there is a criterion for the closedness of $M+N$ only in terms of $P_M$ and $P_N$, namely
 $M+N$ is closed if and only if the sequence $((P_M P_N)^n)_n$ is a Cauchy sequence in norm
(see Remark \ref{rem:alternating_projections}),
there cannot be such a criterion only in terms of $p_M$ and $p_N$.

\begin{exa}
  A concrete example of two closed subspaces $M$ and $N$
  of a Hilbert space $H$ such that $M+N$ is not closed can be obtained as follows (compare the discussion
  preceding Problem 52 in \cite{halmos-hspb}): Take a continuous linear
  operator $T$ on $H$ with non-closed range, and let $M$ be the graph of $T$, that is,
  \begin{equation*}
    M = \{ (x, Tx): x \in H \} \subset H \oplus H.
  \end{equation*}
  Set $N = H \oplus \{0\}$. Then $M$ and $N$ are closed, but
  \begin{equation*}
    M + N = H \oplus \ran(T)
  \end{equation*}
  is not closed. Suppose now that $T$ is additionally self-adjoint and compact. It is
  easy to check that the projection onto $M$ is given by
  \begin{equation*}
    P_M = \begin{pmatrix} (1+T^2)^{-1} & T ( 1+T^2)^{-1} \\ T(1+T^2)^{-1} & T^2 (1+T^2)^{-1}
    \end{pmatrix}.
  \end{equation*}
  Clearly,
  \begin{equation*}
    P_{N} = \begin{pmatrix} 1 & 0 \\ 0 & 0 \end{pmatrix}.
  \end{equation*}
  However, the equivalence classes $p_{M}$ and $p_N$
  of these projections in the Calkin algebra
  are the same. In particular, we see that there cannot be a criterion for the closedness
  of $M+N$ only in terms of $p_M$ and $p_N$.

  Moreover, $M \cap N = \ker(T) \oplus \{0\}$, so
  \begin{equation*}
    P_{M \cap N} = \begin{pmatrix} P_{\ker(T)} & 0 \\ 0 & 0 \end{pmatrix}.
  \end{equation*}
  Hence, if both $\ker(T)$ and $H \ominus \ker(T)$ are infinite dimensional,
  $p_{M \cap N}$ does not belong to the unital $C^*$-algebra generated by $p_M$ and
  $p_N$. For a concrete example, set $H = \ell^2(\N)$, choose a null sequence
  $(a_n)_n$ of real numbers with infinitely many zero and infinitely many non-zero terms,
  and let $T$ be componentwise multiplication with $(a_n)_n$.
\end{exa}

In the presence of conditions \ref{it:cond_a} and \ref{it:cond_b}, the situation is better.
\begin{lem}
  \label{lem:proj_in}
  Under the above hypotheses,
  $p_{(M_1 + \ldots + M_{r-1}) \cap M_r} \in \A$. Moreover, if $\pi = \bigoplus_{i=1}^r \pi_i$
  is the faithful representation from Proposition \ref{prop:special_representation}, we have
  \begin{equation*}
    \pi_i (p_{(M_1 + \ldots + M_{r-1}) \cap M_r}) = \pi_i (p_{S_i \cap M_r}),
  \end{equation*}
  for $i=1,\ldots,r-1$, and $\pi_r(p_{(M_1 + \ldots + M_{r-1})\cap M_r})=0$.
\end{lem}

\begin{proof}
  If $r=2$, condition \ref{it:cond_b} asserts that $P_{M_1} P_{M_2}$ is a compact operator. Since $P_{M_1 \cap M_2}
  = P_{M_1} P_{M_2} P_{M_1 \cap M_2}$, we conclude that $p_{M_1 \cap M_2} = 0$, so the statement is
  trivial for $r = 2$.

  Now, let us assume that $r \ge 3$ and define
  \begin{equation*}
    S = M_1 + \ldots + M_{r-1}.
  \end{equation*}
  In a first step, we show that the sequence $((p_{M_r} p_{S} p_{M_r})^n)_n$
  converges to an element $q_u \in \A$ with $q_u \ge p_{S \cap M_r}$.
  To this end, let $\pi = \bigoplus_{i=1}^r \pi_i$ be the faithful
  representation from Proposition \ref{prop:special_representation}.
  By Lemma \ref{lem:special_rep_proj_sum}, we have $\pi_i (p_S) = \pi_i(p_{S_i})$ for each $i$.
  Since $S_i + M_r$ is closed, Lemma
  \ref{it:sum_closed_spectrum_and_algebra_membership_intersection}
  shows that for $i=1,\ldots,r-1$,
  \begin{equation}
    \label{eqn:lem_proj_in_1}
    \pi_i \big( (p_{M_r} p_{S} p_{M_r})^n \big)
    = \pi_i \big( (p_{M_r} p_{S_i} p_{M_r})^n \big) \xrightarrow{n \to \infty}
    \pi_i (p_{S_i \cap M_r}).
  \end{equation}
  Clearly, $\pi_r (p_{M_r} p_S p_{M_r}) = 0$.
  Since $\pi = \bigoplus_{i=1}^r \pi_i$ is a faithful representation,
  we conclude that $ ((p_{M_r} p_S p_{M_r})^n)_n$ is a Cauchy sequence in $\A$.
  Denoting its limit by $q_u$, we see from
  \begin{equation*}
    (p_{M_r} p_S p_{M_r})^n - p_{S \cap M_r}
    = (p_{M_r} p_S p_{M_r} - p_{S \cap M_r})^n \ge 0
  \end{equation*}
  for all $n \in \N$
  that $q_u \ge p_{S \cap M_r}$.

  The next step is to prove that $0$ is not a cluster point of the spectrum of the positive
  element $a=p_{S_1 \cap M_r} + \ldots + p_{S_{r-1} \cap M_r} \in \A$, and that
  \begin{equation*}
    q_l = \chi_{(0,\infty)} (a) \le p_{S \cap M_r}.
  \end{equation*}
  To this end, we fix an $i \in \{1,\ldots,r-1\}$, and for $j=1,\ldots,r-1$ with $j \neq i$,
  we set
  \begin{equation*}
    N_j = M_1 + \ldots + \widehat M_i + \ldots + \widehat M_j + \ldots + M_{r-1} \subset S_i,
  \end{equation*}
  which is understood as the zero vector space if $r=3$. Clearly, $N_j$ is closed by condition \ref{it:cond_a}.
  Then $p_{N_j} \in \A$, and
  just as in the proof of Lemma \ref{lem:special_rep_proj_sum}, we see that
  $\pi_i(p_{S_j}) = \pi_i(p_{N_j})$.
  Since $N_j + M_r$ and $S_j + M_r$ are closed by condition \ref{it:cond_a}, an application of
  Lemma \ref{it:sum_closed_spectrum_and_algebra_membership_intersection} yields that $p_{N_j \cap M_r}$ belongs to $\A$
  and that
  $\pi_i( p_{S_j \cap M_r}) = \pi_i(p_{N_j \cap M_r})$.
  Therefore,
  \begin{equation*}
    \pi_i(a)
    = \pi_i (p_{N_{1} \cap M_r} + \ldots + p_{N_{i-1} \cap M_r} + p_{S_i \cap M_r} +
    p_{N_{i+1} \cap M_r} + \ldots + p_{N_{r-1} \cap M_r}).
  \end{equation*}
  Using the fact that the algebraic sum
  \begin{equation*}
    N_1 \cap M_r + \ldots + N_{i-1} \cap M_r + S_i \cap M_r + N_{i+1} \cap M_r + \ldots + N_{r-1} \cap M_r
  \end{equation*}
  equals $S_i \cap M_r$ and is therefore evidently closed,
  we conclude with the help of Lemma
  \ref{it:sum_closed_spectrum_and_algebra_membership_spectrum}
  that $0$ is not a cluster point of $\sigma( \pi_i(a))$, and that
  \begin{equation}
    \label{eqn:lem_proj_in_2}
    \chi_{(0,\infty)} (\pi_i(a)) = \pi_i (p_{S_i \cap M_r}).
  \end{equation}
  Since
  $\pi_r(a) = 0$, and since $\pi = \bigoplus_{i=1}^r \pi_i$ is a faithful
  representation of $\A$, it follows that $0$ is not a cluster point of $\sigma(a)$.
  Thus, we can define
  \begin{equation*}
    q_l = \chi_{(0,\infty)}(a) \in \A.
  \end{equation*}
  To prove the asserted inequality,
  we note that $a \le (r-1) \, p_{S \cap M_r}$, and that
  $a$ and $p_{S \cap M_r}$ commute. Hence Lemma
  \ref{it:C_alg_gelfand_consequences_increasing} shows that
  \begin{equation*}
    q_l \le \chi_{(0,\infty)} ( (r-1) \, p_{S \cap M_r})
    = p_{S \cap M_r}.
  \end{equation*}

  We have established the following situation so far:
  \begin{equation*}
    q_l \le p_{(M_1 + \ldots + M_{r-1}) \cap M_r} \le q_u,
  \end{equation*}
  and $q_l$ and $q_u$ belong to $\A$. We now finish the proof of $p_{(M_1 + \ldots + M_{r-1}) \cap M_r} \in \A$ by showing that $q_l = q_u$.
  Using once again the representation from Proposition \ref{prop:special_representation},
  it suffices to show that $\pi_i(q_l) = \pi_i(q_u)$ for $i=1,\ldots,r$.
  This is obvious for $i=r$, because $\pi_r(q_l) = 0 = \pi_r(q_u)$.
  So let $i \in \{1,\ldots,r-1\}$. According to equation \eqref{eqn:lem_proj_in_1},
  we have $\pi_i(q_u) = \pi_i(p_{S_i \cap M_r})$, while equation \eqref{eqn:lem_proj_in_2}
  shows that $\pi_i(q_l) = \pi_i (p_{S_i \cap M_r})$, as desired. The additional assertion is now obvious.
\end{proof}

We are now in the position to prove the main theorem of this section.
\begin{thm}
  \label{thm:sum_closed}
  Let $H$ be a Hilbert space, let $r \ge 2$ and let $M_1, \ldots ,M_r \subset H$ be closed subspaces
  such that the following two conditions hold:
  \begin{enumerate}[label=\normalfont{(\alph*)}]
  \item Any algebraic sum of $r-1$ or fewer subspaces of the $M_i$ is closed, that is, for any subset
  $\{i_1,\ldots,i_k\} \subset \{1,\ldots,r\}$ with $k \le r-1$,
  the sum
  \begin{equation*}
    M_{i_1} + \ldots + M_{i_k}
  \end{equation*}
  is closed.
  \item Any product of the $P_{M_i}$ containing each $P_{M_i}$ at least once is compact, that is,
    for any collection of (not necessarily distinct) indices $i_1,\ldots,i_k$ with $\{i_1,\ldots,i_k \}
    = \{1,\ldots,r\}$, the operator
    \begin{equation*}
      P_{M_{i_1}} P_{M_{i_2}} \ldots P_{M_{i_k}}
    \end{equation*}
    is compact.
  \end{enumerate}
  Then the algebraic sum $M_1 + \ldots + M_r$ is closed.
\end{thm}

\begin{proof}
  As above, let $\A$ be the unital $C^*$-algebra generated by $p_{M_1}, \ldots, p_{M_r}$,
  and let
  $\pi = \bigoplus_{i=1}^r \pi_i$ be the faithful representation from Proposition
  \ref{prop:special_representation}.
  By the discussion preceding Lemma \ref{lem:special_rep_proj_sum}, the elements $p_{S_i}$ and $p_{S_i \cap M_r}$,
  as well as $p_{M_1 + \ldots + M_{r-1}}$, all belong to $\A$ for $i=1,\ldots,r-1$.
  According to Lemma \ref{lem:proj_in}, this is also true for $p_{(M_1 + \ldots + M_{r-1}) \cap M_r}$,
  and
  \begin{equation*}
    \pi_i (p_{(M_1 + \ldots + M_{r-1}) \cap M_r}) = \pi_i(p_{S_i \cap M_r})
    \quad \text{ for } i=1,\ldots,r-1.
  \end{equation*}
  Moreover, for these $i$, we have $\pi_i (p_{M_1 + \ldots + M_{r-1}}) = \pi_i (p_{S_i})$ by Lemma \ref{lem:special_rep_proj_sum}.
  Combining these results, we obtain
  \begin{align*}
    ||\pi_i(p_{M_1 + \ldots + M_{r-1}} p_{M_r} - p_{(M_1 + \ldots + M_{r-1}) \cap M_r})||
    &= ||\pi_i(p_{S_i} p_{M_r} - p_{S_i \cap M_r})|| \\
    &\le \fangle_e(S_i,M_r).
  \end{align*}
  Since $\pi_r (p_{M_r}) = 0 = \pi_r(p_{(M_1 + \ldots + M_{r-1})\cap M_r})$, we conclude that
  \begin{align*}
    \fangle_e(M_1 + \ldots +M_{r-1}, M_r)
    &= ||p_{M_1 + \ldots + M_{r-1}} p_{M_r} - p_{(M_1 + \ldots + M_{r-1}) \cap M_r}|| \\
    &\le \max_{1 \le i \le r-1} \fangle_e(S_i,M_r) < 1
  \end{align*}
  because $S_i + M_r$ is closed for each $i$ by condition \ref{it:cond_a}.
\end{proof}

The desired result about sums of Fock spaces follows now by a straightforward inductive argument.
\begin{cor}
  \label{cor:Fock_sum_closed}
  Let $V_1, \ldots, V_r \subset \C^d$ be subspaces. Then the algebraic sum
  \begin{equation*}
    \F(V_1) + \ldots + \F(V_r) \subset \F(\C^d)
  \end{equation*}
  is closed.
\end{cor}

\begin{proof}
  We prove the result by induction on $r$, noting that the case $r=1$ is trivial. So suppose that $r \ge 2$ and
  that the assertion has been proved for $k \le r-1$.
  In order to show that sums of $r$ Fock spaces $\F(V_1) , \ldots, \F(V_r)$ are closed,
  it suffices to consider the case where
  \begin{equation*}
    V_1 \cap \ldots \cap V_r = \{0\}
  \end{equation*}
  by Lemma \ref{cor:Fock_closed_reduction}.
  Let $M_i = \F(V_i)$ for each $i$.
  As an application of Lemma \ref{lem:proj_product_compact}, we see that
  condition \ref{it:cond_b} of the preceding theorem is satisfied,
  whereas condition \ref{it:cond_a} holds by the inductive hypothesis. Thus the assertion follows from the preceding theorem.
\end{proof}

In the terminology of the second section, this result, combined
with Lemma \ref{lem:full_Fock_good}, shows that every radical homogeneous ideal is \admideal.
Hence,
Proposition \ref{prop:good_bounded_maps} and Corollary \ref{cor:good_algebra_iso} hold
without the additional hypotheses on $I$ and $J$.
We thus obtain the following generalization of \cite[Theorem 8.5]{davramshal}.

\begin{thm}
  Let $I$ and $J$ be radical homogeneous ideals in $\polyring{d}$ and $\polyring{d'}$, respectively.
  The algebras $\A_I$ and $\A_J$ are isomorphic if and only if there exist linear maps
  $A: \C^{d'} \to \C^{d}$ and $B: \C^{d} \to \C^{d'}$ which restrict to mutually inverse
  bijections $A: Z(J) \to Z(I)$ and $B: Z(I) \to Z(J)$. \qed
\end{thm}

\begin{rem*}
  Using Corollary \ref{cor:good_algebra_iso} in place of \cite[Theorem 7.17]{davramshal},
  we also see that the hypothesis of the ideals being tractable can be removed from Corollary 9.7
  and Theorem 11.7 (b) in \cite{davramshal}.
\end{rem*}

\subsection*{Acknowledgements} The author wishes to thank Ken Davidson
for valuable discussions, and for the kind hospitality provided during
the author's stay at the University of Waterloo.
Moreover, he is grateful to his Master's thesis advisor J\"org Eschmeier for his advice and support.

\end{document}